\documentclass[11pt,a4paper,reqno]{amsart}
\usepackage{a4wide}
\usepackage[utf8x]{inputenc}
\usepackage{ucs}
\usepackage{amsmath,amssymb,paralist,bbm,amsthm}

\usepackage{graphicx}

\title[Bifurcation and Hausdorff dimension in families of chaotically driven maps]{Bifurcation and Hausdorff dimension in families of chaotically driven maps with multiplicative forcing}

\author{Gerhard Keller}
\address{Department Mathematik, Universität Erlangen-Nürnberg, Cauerstr. 11, 91058 Erlangen, Germany}
\email{keller@mi.uni-erlangen.de}
\thanks{This work is funded by DFG grant Ke 514/8-1. The authors thank B. Saussol and J. Schmeling for valuable help on some questions concerning the multifractal analysis of Birkhoff averages.}
\author{Atsuya Otani}
\date{\today}
\email{otani@mi.uni-erlangen.de}

\newcommand{\assign}{:=}
\newcommand{\cdummy}{\cdot}

\newcommand{\longrightarrowlim}{\mathop{\longrightarrow}\limits}

\newcommand{\nocomma}{}
\newcommand{\nosymbol}{}

\newcommand{\tmop}[1]{\ensuremath{\operatorname{#1}}}

\newcommand{\fs}{\,.}
\newcommand{\com}{\,,}

\newtheorem{corollary}{Corollary}

\newtheorem{lemma}{Lemma}
\newtheorem{theorem}{Theorem}
\newtheorem{proposition}{Proposition}
\theoremstyle{definition}

\newtheorem{remark}{Remark}

\newcommand{\R}{\mathbbm{R}}
\newcommand{\N}{\mathbbm{N}}

\begin{document}

\begin{abstract}
We study bifurcations of invariant graphs in skew product dynamical systems driven by hyperbolic surface maps $T$ like Anosov surface diffeomorphisms or baker maps and with one-dimensional concave fibre maps under multiplicative forcing when the forcing is scaled by a parameter $r=e^{-t}$. For a range of parameters two invariant graphs (a trivial and a non-trivial one) coexist, and we use thermodynamic formalism to characterize the parameter dependence of the Hausdorff and packing dimension of the set of points where both graphs coincide.
As a corollary we characterize the parameter dependence of the dimension of the global attractor $\mathcal{A}_t$: Hausdorff and packing dimension have a common value $\dim(\mathcal{A}_t)$, and there is a critical parameter $\gamma_c^-$ determined by the SRB measure of $T^{-1}$ such that 
$\dim(\mathcal{A}_t)=3$ for $t\leqslant\gamma_c^-$ and $t\mapsto\dim(\mathcal{A}_t)$ is strictly decreasing for $t\in[\gamma_c^-,\gamma_{\max})$.
\end{abstract}

\maketitle

\section{The general setting and a review of main results}\label{sect:pre}

In this paper we study bifurcations in skew product dynamical systems driven by a basis dynamical system $(\Theta, \mathcal{B}, T)$, where $( \Theta,
\mathcal{B})$ is a measurable space and $T : \Theta \rightarrow \Theta$ a
bi-measurable map. We denote the set of $T$-invariant probability measures and
its subset of ergodic measures by $\mathcal{P}_T(\Theta)$ and
$\mathcal{E}_T( \Theta)$, respectively. For the sake of simplicity, we
will use the notation $\mu \left( \psi \right) \assign \int_\Theta \psi\, d \mu$ for $\mu \in \mathcal{P}_T ( \Theta )$ and $\psi \in C ( \Theta ; \mathbbm{R})$. We also denote $\mathbbm{R}_{\geqslant} \assign [0,\infty)$.

\subsection{The skew product system}

For each parameter $t \in \mathbbm{\mathbbm{R}}$ 
we define a skew-product transformation
\[ T_t : \Theta \times \mathbbm{R}_{\geqslant} \rightarrow \Theta \times \mathbbm{R}_{\geqslant},
\quad T_t \left( \theta, x \right) \assign \left( T  \theta,
f_t \left( \theta, x \right)  \right),
\]
with a fibre function
\[ f_t : \Theta \times \mathbbm{R}_{\geqslant} \rightarrow \mathbbm{R}_{\geqslant},
\quad f_t  (\theta, x) \assign e^{- t} g (\theta) h (x)
\]
where
\begin{enumerate}[\quad$\triangleright$]
  \item $h \in C^1 ( \mathbbm{\mathbbm{R}}_{\geqslant} ; \mathbbm{R})$ is strictly concave with $h (0) =
  0$, $h' (x) > 0 $ for $x > 0$, $h' (0) = 1$, and $\lim_{x\to\infty}\frac{h(x)}{x}=0$,
  
  \item $g:\Theta\to(0,\infty)$ is bounded and measurable.\end{enumerate}
For $n \geqslant 2$ we
define iteratively
$f_t^n: \Theta \times \mathbbm{R}_{\geqslant} \rightarrow\mathbbm{R}_{\geqslant}$ and
$T_t : \Theta \times \mathbbm{R}_{\geqslant} \rightarrow \Theta \times \mathbbm{R}_{\geqslant}$,
\[ 
f_t^n \left( \theta, x \right) \assign f_t \left( T^{n - 1} \theta, f_t^{n - 1}
   \left( \theta,x \right) \right) \hspace{1em} \text{\tmop{and}} \hspace{1em}
   T_t^n \left( \theta, x \right) \assign T_t \left( T_t^{n - 1} \left( \theta, x
   \right) \right)\fs  \]
\begin{remark}
  \label{f_calculation}The following properties are easily verified:
  \begin{enumerate}[a)]
    \item $T_t^n \left( \theta, x \right) = \left( T^n \theta, f_t^n \left( \theta,x
    \right) \right)$ for $\left( \theta, x \right) \in \Theta \times
    \mathbbm{R}_{\geqslant}$ and $_{} n \in \mathbbm{N}_0$.
    
    \item $\frac{d}{d x} f_t^n \left( \theta, x \right) > 0$ and $\frac{d^2}{d
    x^2} f_t^n \left( \theta, x \right) < 0$ for all $\left( \theta, x \right)
    \in \Theta \times ( 0 , \infty ) $ and $n \in \mathbbm{N}$.
    \item For each $t\in\R$ there is $M_t>0$ such that 
    $f_t( \theta, M_t)<M_t$ for all $\theta\in\Theta$.
  \end{enumerate}
\end{remark}

\subsection{The maximal invariant function $\varphi_t$ and its zero set $N_t$}

A function $\varphi \nosymbol : \Theta \rightarrow \mathbbm{R}$ is invariant (or
more precisely $T_t$-invariant), if
\[ T_t \left( \theta, \varphi ( \theta ) \right) = \left( T \theta,
   \varphi ( T \theta ) \right) \quad
   \text{or, equivalently,}\quad f_t \left( \theta, \varphi(\theta) \right) = \varphi \left( T \theta \right), \]
for all $\theta \in \Theta$. The function $\varphi\equiv0$ is always invariant. We call its graph $\Phi_{\text{base}}=\{(\theta,0):\theta\in\Theta\}$ the \emph{baseline} of the skew product system.

Since our fibre maps are monotone and strictly concave, this skew-product system possesses at most two essentially different measurable invariant functions, among them the maximal one, as the following lemma shows.

\begin{lemma}  \label{max_invariant_function}
\begin{enumerate}[a)]
\item For $t \in \mathbbm{R}$ the \emph{maximal
  $T_t$-invariant function} $\varphi_t : \Theta \rightarrow \mathbbm{R}_\geqslant$,
  \begin{equation}\label{eq:varphi-def}
\varphi_t ( \theta) \assign\lim_{n \rightarrow \infty} \psi_{t, n}
     ( \theta ) = \inf_n \psi_{t, n} ( \theta) \com
  \end{equation}
  is well defined
  where 
  $\psi_{t, n} \left( \theta \right) \assign f_t^n \left( T^{- n} \theta, M_t
  \right)$. It is indeed maximal, i.e. for every $T_t$-invariant function $\varphi$ we have
  that $0\leqslant\varphi \left( \theta \right) \leqslant \varphi_t \left( \theta
  \right)$ for all $\theta \in \Theta$. Its graph is denoted by
  $\Phi_t \assign \left\{ \left( \theta, \varphi_t \left( \theta \right)
  \right) : \theta \in \Theta \right\}$.
  \item
  Let $\varphi$ be a measurable $T_t$-invariant function. Then we have for every
  $\mu \in \mathcal{E}_T ( \Theta)$
\begin{equation}
\varphi=0\text{ $\mu$-a.e.\quad or\quad }\varphi=\varphi_t\text{ $\mu$-a.e.}
\end{equation}
\end{enumerate}
\end{lemma}
The proof of part a) of this lemma is identical to the one in \cite[pp.144-145]{Keller1996}, while part b) is contained in \cite[Lemma 1]{Keller1996}. 
Observe that in that reference the base system is an irrational rotation on $\mathbb{T}^1$, but only the invertibility and the ergodicity of the invariant (Lebesgue) measure are  used for the proofs.

Depending on the stability properties of the fibre maps at $x=0$ relative to a measure $\mu\in\mathcal{E}_T ( \Theta)$, the maximal invariant function $\varphi_t$ may be identical to zero, 
strictly positive, or - and this is the most interesting case - it may have zeros without being identical to zero. 

In this note we describe the measure theoretical and topological properties of the sets
\begin{equation}
N_t \assign \left\{ \theta \in \Theta : \varphi_t (\theta) = 0 \right\},
\end{equation}
quantify the size of these sets
in terms of their dimension and study the dependence of the dimension on the parameter $t$.

\begin{remark}\label{remark:basic1}
  The following properties are immediate consequences of the definitions.
  \begin{enumerate}[a)]
    \item $\varphi_t$ is measurable. In particular, $N_t \in \mathcal{B}$.
    
    \item $N_t$ is invariant under $T$, i.e. $T \left( N_t \right) = N_t
    = T^{- 1} \left( N_t \right)$.
    
    \item\label{remark:basic1-3} For $t < s$ we have $\varphi_t \left( \theta \right) \geqslant
    \varphi_s \left( \theta \right)$ for all $\theta \in \Theta$, whence $N_t
    \subseteq N_s$.
	We say therefore that the family $ ( N_t ) _{t \in
    \mathbbm{R}}$ is a filtration.
  \end{enumerate}
\end{remark}

\begin{remark} 
The set $\mathcal{A}_t\assign\{(\theta,x):0\leqslant x\leqslant\varphi_t(\theta)\}$
is the global attractor of the map $T_t$. It is bounded from above by the \emph{upper bounding graph} $\Phi_t$.
 As each ergodic $T_t$-invariant probability measure is
    supported by an invariant graph \cite[Theorem 1.8.4(iv)]{Arnold-book}, Lemma~\ref{max_invariant_function} shows that it is supported by $\Phi_t$ or by the baseline $\Phi_{\text{base}}$.

\end{remark}

In view of the filtration property of the sets $ ( N_t ) _{t \in
\mathbbm{R}}$ it is natural to define, for each $\theta\in\Theta$, a critical parameter $t_c(\theta)$ by
\begin{equation}
t_c(\theta) \assign \inf\left\{t\in\R: \varphi_t(\theta)=0\right\}.
\end{equation}
Because of Remark~\ref{remark:basic1}\ref{remark:basic1-3}) we have 
$t_c(\theta)=\sup\{t\in\R: \varphi_t(\theta)>0\}$, and for each $t\in\R$
\begin{equation}
S_t\assign\{\theta\in\Theta: t_c(\theta)=t\} = \bigcap_{t'' > t} N_{t''} \setminus \bigcup_{t' < t} N_{t'}.
\end{equation}
As the baseline itself is a trivial invariant graph, these points
can be understood as bifurcation points of invariant graphs.

\subsection{The plan of this note}
In section~\ref{subsec:fibre-wise} we characterize the sets $N_t$ and $S_t$ in terms of Birkhoff averages
\begin{equation}
\Gamma(\theta)\assign\liminf_{n\to\infty}\frac{1}{n}\sum_{k=1}^n\log g(T^{-k}\theta)
\end{equation}
which are closely related to the system's fibre-wise lower backwards Lyapunov exponents at the baseline. The main results are:
\begin{enumerate}[\quad$\triangleright$]
\item $\{\Gamma<t\}\subseteq N_t\subseteq\{\Gamma\leqslant t\}$ and $S_t=\{\Gamma=t\}$.
\item $\mu(S_t\setminus N_t)=0$ for each $\mu\in\mathcal{P}_T(\Theta)$.
\end{enumerate}

In section~\ref{subsec:average} we characterize the same sets for each $\mu\in\mathcal{E}_T(\Theta)$ in terms of the averaged quantity
\begin{equation}
\gamma(\mu)\assign \int\log g\,d\mu \fs 
\end{equation}
By Birkhoff's ergodic theorem, $\gamma(\mu)=\Gamma(\theta)$ for $\mu$-a.e. $\theta$.
The main observations are
\begin{enumerate}[\quad$\triangleright$]
\item $\mu(N_t)=1$ if and only if $\gamma(\mu)\leqslant t$ and
\item $\mu(S_t)=1$ if and only if $\gamma(\mu)=t$.
\end{enumerate}

Finally, in section~\ref{sec:hyperbolic}, we determine the Hausdorff dimensions $\dim_H$ and packing dimensions $\dim_P$ of the sets $N_t$ and $S_t$ for topologically mixing Anosov surface diffeomorphisms and baker maps using thermodynamic formalism.
Define
\begin{equation}
\gamma_{\min} \assign \inf_{\mu \in \mathcal{P}_T(\Theta)} \gamma \left(
   \mu \right) \hspace{1em} \text{\tmop{and}} \hspace{1em} \gamma_{\max} \assign
   \sup_{\mu \in \mathcal{P}_T( \Theta)} \gamma \left( \mu \right) .
\end{equation}
 In the Anosov case the main result reads:
suppose $\gamma_{\min}<\gamma_{\max}$ and denote by ${\gamma_c^-}\assign\gamma(\mu_{\mathrm{SRB}}^-)$ the average exponent of the SRB meaure of $T^{-1}$. There is a real analytic function $D : \left( \gamma_{\min}, \gamma_{\max} \right) \rightarrow
  \left[ 0, 1 \right]$ such that $D \left( {\gamma_c^-}
  \right) = 1, D'' \left( {\gamma_c^-} \right) < 0 \text{}$,
    \[ D' \left( t \right) = \left\{ \begin{array}{ll}
       > 0 & \mbox{for } t \in \left( \gamma_{\min}, {\gamma_c^-} \right) \\
       < 0 & \mbox{for } t \in \left( {\gamma_c^-}, \gamma_{\max} \right) 
     \end{array} \right. , \qquad\text{and}\]
\begin{enumerate}[\quad$\triangleright$]
\item 
       $\dim_H \left( N_t \right) = \dim_H \left( S_t \right) = D \left( t
       \right) + 1$ for $t \in \left( \gamma_{\min}, {\gamma_c^-}
       \right)$,
\item
       $\dim_H ( \Theta \setminus N_t ) = \dim_H ( S_t ) =       
       \dim_P ( \Theta \setminus N_t ) = D ( t ) + 1$
       for $t \in \left( {\gamma_c^-}, \gamma_{\max} \right)$, and
\item       
       $\dim _P ( N _t ) = 2 > \dim _H ( N _t ) $
for $t \in \left( \gamma_{\min}, {\gamma_c^-} \right)$.
\end{enumerate}

A number of proofs are deferred to section~\ref{sec:proofs}.

\section{Characterization of the sets $N_t$ and $S_t$ in terms of Lyapunov exponents}\label{sec:characterization}

\subsection{The sets $N_t$ and $S_t$ via fibre-wise Lyapunov exponents}
\label{subsec:fibre-wise}

Recall that $\varphi_t$ is defined in (\ref{eq:varphi-def}) as a pullback limit. Therefore it is natural to characterize its zeros in terms of the fibre-wise lower backwards Lyapunov exponents at the baseline
\begin{equation}
\Gamma_t(\theta)\assign
\liminf_{n\to\infty}\frac{1}{n}\log\left|\frac{d}{dx}f_t^n(T^{-n}\theta,x)_{|x=0}\right|
=
\Gamma(\theta)-t\fs 
\end{equation}

  The following characterization of the set $S_t$ in terms of $\Gamma(\theta)$ is an essential point of this note. Under additional hyperbolicity assumptions it will be the key
  to a multifractal bifurcation analysis of the family $(\varphi_t )_{t\in\R}$.

\begin{theorem}[$N_t$ and $S_t$ via trajectory-wise Lyapunov exponents]
\label{theo:main-Lambda}
	Let $t \in \mathbbm{R}$ and $\theta \in \Theta$.
  \begin{enumerate}[a)]
    \item If $\theta\not\in N_t$, then $\Gamma( \theta) \geqslant t$, i.e. $\Gamma_t(\theta)\geqslant0$.\label{theo:main-Lambda-a}
    \item If $\theta\in N_t$, then $\Gamma( \theta) \leqslant t$, i.e.  $\Gamma_t(\theta)\leqslant0$.  \label{theo:main-Lambda-b}  
    \item $\Gamma(\theta)=t_c(\theta)$ and $S_t=\{\theta\in\Theta: \Gamma(\theta)=t\}=\{\theta\in\Theta: \Gamma_t(\theta)=0\}$.\label{theo:main-Lambda-c}
  \end{enumerate}
\end{theorem}
\noindent
Although we have no proof, we do not believe that $S_t\subseteq N_t$. Instead we have the following characterization of points in $S_t\setminus N_t$.
\begin{proposition}[Characterization of $S_t\setminus N_t$]
\label{prop:SN}
Let $t\in\R$ and $\theta\in\Theta$. Then $\theta\in S_t\setminus N_t$
if and only if 
\begin{equation}
  \sum_{k = 1}^{\ell_i} \varphi_{t} (T^{- k} \theta)
  +|\log \varphi_{t}(T^{-\ell_i}\theta)|
  =
  o(\ell_i)
\end{equation}
along some subsequence $(\ell_i)_{i\geqslant1}$.
\end{proposition}
\noindent
The \emph{proofs} of Theorem~\ref{theo:main-Lambda} and Proposition~\ref{prop:SN}
are provided in section \ref{subsec:proofs1}.

\begin{corollary}\label{coro:St-Nt}
$\mu(S_t\setminus N_t)=0$ for each $\mu\in\mathcal{P}_T(\Theta)$ and $t\in\R$.
\end{corollary}
\begin{proof}
As $S_t\setminus N_t$ is $T$-invariant and as $0\leqslant\varphi_t\leqslant M_t$, Proposition~\ref{prop:SN} implies
\begin{equation}
\int_{S_t\setminus N_t}\varphi_t\,d\mu
=
\lim_{i\to\infty}\int_{S_t\setminus N_t}\frac{1}{\ell_i}\sum_{k = 1}^{\ell_i} \varphi_{t}\circ T^{- k}\,d\mu 
=0 \com
\end{equation}
and as $\varphi_t>0$ on $S_t\setminus N_t$, it follows that $\mu(S_t\setminus N_t)=0$.
\end{proof}

\subsection{The sets $N_t$ and $S_t$ via average Lyapunov exponents}
\label{subsec:average}

With respect to any invariant measure $\mu \in \mathcal{P}_T( \Theta)$ we define the average fibre-wise Lyapunov exponent at the baseline by
\begin{equation}
\gamma_t \left( \mu \right) \assign \int \log \left| \frac{d}{d
x} f_t \left(\,\cdummy\,, x \right) _{|x=0} \right| d \mu
=\gamma(\mu)-t \fs 
\end{equation} 
We note that 
\begin{equation}
 \int\Gamma(\theta)\,d\mu(\theta)=\int\log g\,d\mu =\gamma(\mu) 
\end{equation}
follows from Birkhoff's ergodic theorem, because $|\log g|$ is bounded by assumption.
\begin{corollary}[$N_t$ and $S_t$ via average Lyapunov exponents]
  \label{coro:N_crit}
  For $\mu \in \mathcal{P}_T( \Theta)$ and $t\in\R$ we have:
  \begin{enumerate}[a)]
    \item If $\mu (N_t) = 0$, then $\gamma_t (\mu) > 0$, i.e. $\gamma( \mu) > t$. \label{coro:N_crit-a}
    \item If $\mu (N_t) = 1$, then $\gamma_t (\mu) \leqslant 0$, i.e. $\gamma( \mu) \leqslant t$. \label{coro:N_crit-b}
  \end{enumerate}
  For ergodic $\mu$ both implications are equivalences and, furthermore,
  \begin{enumerate}[a)]
  \setcounter{enumi}{2}
  \item \label{coro:N_crit-c} $\mu \left( S_t \right)
  = 1$ for $t = \gamma \left( \mu \right)$, and $\mu \left( S_t \right) = 0$
  otherwise. 
  \end{enumerate}
\end{corollary}

\begin{proof}
a) If $\theta\in\Theta\setminus N_t$ for $\mu$-a.e.\ $\theta$, then $\Gamma_t(\theta)\geqslant0$ for $\mu$-a.e.\ $\theta$ by Theorem~\ref{theo:main-Lambda}\ref{theo:main-Lambda-a}), 
and as $\mu(S_t)=\mu(S_t\setminus N_t)=0$ by Corollary~\ref{coro:St-Nt}, $\Gamma_t(\theta)>0$ for $\mu$-a.e.\ $\theta$ by Theorem~\ref{theo:main-Lambda}\ref{theo:main-Lambda-c}). Hence $\gamma_t(\mu)>0$.
b) If $\theta\in N_t$ for $\mu$-a.e.\ $\theta$, then $\Gamma_t(\theta)\leqslant 0$ for $\mu$-a.e.\ $\theta$ by Theorem~\ref{theo:main-Lambda}\ref{theo:main-Lambda-b}), so that $\gamma_t(\mu)\leqslant0$.

If $\mu$ is ergodic, the cases a) and b) are exhaustive so that we have equivalences, and claim~c)
  follows from Theorem~\ref{theo:main-Lambda}\ref{theo:main-Lambda-c}), since $\Gamma_{} \left(
  \theta \right) = \gamma \left( \mu \right)$ for $\mu$ a.e.\ $\theta \in \Theta$
  due to Birkhoff's ergodic theorem.
\end{proof}

Let $\Gamma^{(n)}(\theta)\assign\frac{1}{n}\sum_{k=1}^n\log g(T^{-k}\theta)$ so that $\Gamma(\theta)=\liminf_{n\to\infty}\Gamma^{(n)}(\theta)$.
  Inspired by the representation of $S_t$ as $\{\theta\in\Theta:\Gamma(\theta)=t\}$ from Theorem~\ref{theo:main-Lambda}\ref{theo:main-Lambda-c}) we define the sets
\begin{equation}
S_t' \assign \left\{ \theta \in \Theta : \lim_{n \rightarrow \infty} \Gamma^{(n)}
     \left( \theta \right) = t \right\}
\end{equation}
and
\begin{equation}
R_t' \assign \left\{
     \theta \in \Theta : \lim_{k \rightarrow \infty} \Gamma^{(n_k)} \left( \theta
     \right) = t \text{ for some } ( n_k ) _k \subset
     \mathbbm{N} \right\} .
\end{equation}
Observe that $S_t'\subseteq S_t\subseteq R_t'$.
Both of these sets are very close to $S_t$ in the following measure-theoretic sense.

\begin{corollary}
  \label{coro:S_ident}
  For every $\mu \in \mathcal{E}_T( \Theta)$ and $t\in\R$ we have
  $\mu \left( S_t \setminus S_t' \right) = 0$ and $\mu \left(R_t'\setminus S_t \right) = 0$.
\end{corollary}

\begin{proof}
From Theorem~\ref{theo:main-Lambda}\ref{theo:main-Lambda-c}), Corollary~\ref{coro:N_crit}\ref{coro:N_crit-c}) and Birkhoff's ergodic theorem we conclude that
  $\mu \left( S_t \right) = \mu \left( S_t' \right) = \mu \left( R_t' \right)
  = 1$ for $t = \gamma \left( \mu \right)$ and that $\mu \left( S_t \right) =
  \mu \left( S_t' \right) = \mu \left( R_t' \right) = 0$ for $t \neq \gamma
  \left( \mu \right)$. Thus, the differences are always sets of measure zero.
\end{proof}

From now on let $\Theta$ be a compact metrizable space, $T : \Theta \rightarrow \Theta$ a
homeomorphism, $g:\Theta\to(0,\infty)$ continuous, and $\mathcal{B}$ the Borel $\sigma$-algebra on $\Theta$.
  Then $\varphi_t$ is upper-semi-continuous as an infimum
  of continuous functions. In particular, $N_t$ is a $G_{\delta}$-set.

\begin{corollary}
Under these topological assumptions:
  \begin{enumerate}[a)]
    \item $N_t = \emptyset$ for $t < \gamma_{\min}$.
    
    \item $\emptyset \neq N_t$ and $N_t \neq \Theta$ for $\gamma_{\min} \leqslant t <
    \gamma_{\max}$.
    
    \item $\mu \left( N_{\gamma_{\max}} \right) = 1$ for all $\mu \in
    \mathcal{P}_T( \Theta)$.
    
        \item $N_t = \Theta$ for $t > \gamma_{\max}$.
  \end{enumerate}
\end{corollary}

\begin{proof}
The semi-uniform ergodic theorem \cite[Theorem 1.9]{Sturman/Stark} yields $\gamma_{\min} \leqslant \Gamma(\theta)
  \leqslant \gamma_{\max}$ for all $\theta\in\Theta$, whence a) and d) follow from Theorem \ref{theo:main-Lambda}. Corollary~\ref{coro:N_crit} 
  shows that $N_t = \emptyset$ implies $t < \gamma(\mu)$ for each $\mu\in\mathcal{P}_T(\Theta)$, in particular
   $t<\gamma_{\min}$, and that $N_t =
  \Theta$ implies $t \geqslant \gamma_{\max}$. This shows assertion b).
  Finally, let $\mu\in\mathcal{P}_T(\Theta)$. As $\mu \left(S_{\gamma_{\max}}\setminus N_{\gamma_{\max}}\right)=0$ 
  by Corollary~\ref{coro:St-Nt}, we have
  $\mu\left(N_{\gamma_{\max}}\right)=\mu\left(N_{\gamma_{\max}}\cup S_{\gamma_{\max}}\right)=\mu\left(\bigcap_{t>\gamma_{\max}}N_t\right)=
  \mu\left(\Theta\right)=1$, where we used d) for the third identity. This proves c).
\end{proof}

\begin{remark}
We do not know whether $N_{\gamma_{\max}} = \Theta$.
\end{remark}

\section{Dimensions of the sets $N_t$ and $S_t$ for hyperbolic systems}
\label{sec:hyperbolic}

Let $ \Theta $ be a 2-dimensional compact Riemannian manifold, $T : \Theta \rightarrow \Theta$ a topologically mixing $C^2$-Anosov diffeomorphism, and $g:\Theta\to(0,\infty)$ a Hölder continuous function.
Denote by  $T_{\theta} \Theta = E^s \left( \theta \right) \oplus  E^u \left( \theta \right)$ the splitting of the tangent fibre over $\theta \in \Theta$ into its stable and unstable subspaces, see \cite{Barreira,Bowen} for details.
In the following, the Hausdorff dimension $\dim_H$ and the packing dimension $\dim_P$ are defined w.r.t. the associated Riemannian metric. We refer to
 \cite{Falconer2} for their definitions.

\begin{remark}
As a lower backward ergodic average, $\Gamma:\Theta\to\R$ is constant along unstable manifolds. Therefore the sets $S_t$ and $S_t'$ are unions of unstable manifolds, see Theorem~\ref{theo:main-Lambda}.
\end{remark}

The map $T^{-1}$ has a unique (and hence ergodic) Sinai-Ruelle-Bowen measure $\mu_{\mathrm{SRB}}^- $ characterized by

  \begin{equation}
    h_T ( \mu _{\mathrm{SRB}}^- ) + \mu _{\mathrm{SRB}}^- \left( \log \left\|
    d T \left| E^s \right. \right\| \right) = 
    \sup_{\mu \in \mathcal{P}_T( \Theta )}
    \bigl( h_T ( \mu ) + \mu \left( \log \left\| d T
    \left| E^s \right. \right\| \right)\bigr) = 0 \com \label{eq:SRB}
  \end{equation}
see \cite[section 4B]{Bowen}.
We define the critical parameter 
${\gamma_c^-} \assign \gamma ( \mu_{\mathrm{SRB}}^-)$. It is the "physical" or "observable" critical parameter of the family $(T_t)_{t \in \R} $ in the sense in which the SRB measure is often called a physical or observable measure, see Remark~\ref{rem:SRB} for more details.

\begin{theorem}  \label{theo:D_dim}
Suppose that $\gamma_{\min} < \gamma_{\max}$ so that $\log g$ is not cohomologous to a constant. Then
  
  \[ \begin{array}{ll}
       \dim_H ( N_t ) = \dim_H ( S_t ) = D ( t) + 1 & \mbox{for } t \in \left( \gamma_{\min}, {\gamma_c^-}       \right] ,\\
       \dim_H ( \Theta \setminus N_t ) = \dim_H ( S_t ) =
       \dim_P ( \Theta \setminus N_t ) = D ( t ) + 1 &
       \mbox{for } t \in \left[ {\gamma_c^-}, \gamma_{\max} \right)
     \end{array} \] 
  with
 \begin{equation}
 	 D ( t ) 
 	 \assign
 	  \max \left\{ \frac{h_T \left( \mu \right)}{ \mu
     \left( - \log \left\| d T \left| E^s \right. \right\| \right)} : \mu \in
     \mathcal{P}_T \left( \Theta \right) \mbox{ and } \mu \left( \log g
     \right) = t \right\} .	\label{eq:D_expression}
  \end{equation}
  Furthermore, $D : \left( \gamma_{\min}, \gamma_{\max} \right) \rightarrow
  \left[ 0, 1 \right]$ is a real analytic function with $D \left( {\gamma_c^-}
  \right) = 1$, $ D'' \left( {\gamma_c^-} \right) < 0 $ and
  \[ D' \left( t \right) = \left\{ \begin{array}{ll}
       > 0 & \mbox{for } t \in \left( \gamma_{\min}, {\gamma_c^-} \right) \\
       < 0 & \mbox{for } t \in \left( {\gamma_c^-}, \gamma_{\max} \right)
     \end{array} \right. . \]
  In addition, there is a unique equilibrium state $\mu_t \in \mathcal{E}_T ( \Theta)$ such that $D( t ) = \frac{ h_T ( \mu_t)}{\mu_t ( - \log \| d T | E^s \| )}$,
  $\mu_t( S_t) = 1$ and $\dim_H(\mu_t)=D(t)+1$.
\end{theorem}

\begin{remark}
  In general one cannot expect that $D$ is concave on $\left( \gamma_{\min}, \gamma_{\max}
  \right)$, as is illustrated by
\cite[Proposition 9.2.2]{Barreira}.
\end{remark}

\begin{corollary}
$ \dim _P ( N _t ) = 2 > \dim _H ( N _t ) $ for $t \in \left( \gamma_{\min}, {\gamma_c^-} \right)$.
\end{corollary}
\begin{proof}
  As $ T : \Theta \rightarrow \Theta $ is topologically transitive, the set $N_t$ is a dense $G_{\delta}$-set by Baire's category theorem, unless it is empty.
 Thus, $ \dim _P ( N _t ) = 2 = D ( \gamma _c ^- ) + 1 > \dim _H ( N _t ) $ for $t \in \left( \gamma_{\min}, {\gamma_c^-} \right)$.
\end{proof}

\begin{remark}\label{rem:SRB}
The forward SRB measure $\mu_{\mathrm{SRB}}^+$, defined as in (\ref{eq:SRB}) but with $-\log\left\| d T \left| E^u \right. \right\|$ instead of $\log\left\| d T \left| E^s \right. \right\| $, defines another critical parameter
${\gamma_c^+}\assign \gamma(\mu_{\mathrm{SRB}}^+)$. The forward and backward SRB measure coincide if and only if they are both absolutely continuous w.r.t. the volume measure \cite[section 4C]{Bowen}. Otherwise one can typically expect that ${\gamma_c^+}\neq{\gamma_c^-}$.

If ${\gamma_c^-}<t<{\gamma_c^+}$, then $\dim_H(\Theta\setminus N_t)=D(t)+1<2$. In particular, $N_t\subseteq\Theta$ has full volume such that the global attractor $\mathcal{A}_t\subseteq\Theta\times\R_\geqslant$ has volume $0$. At the same time, the fibre-wise forward Lyapunov exponent at the base satisfies $\lim_{n\to\infty}\frac{1}{n}\log\left|\frac{d}{dx}f_t^n(\theta,x)_{|x=0}\right|=\int\log\left|\frac{d}{dx}f_t(\theta,x)_{|x=0}\right|\,d\mu_{\mathrm{SRB}}^+=\gamma_c^+-t>0$ for a.e. $\theta\in\Theta$ w.r.t. the volume measure \cite[section 4C]{Bowen}.

If ${\gamma_c^+}<t<{\gamma_c^-}$, then $\dim_H(N_t)=D(t)+1<2$. In particular, $N_t\subseteq\Theta$ has zero volume. But now the fibre-wise forward Lyapunov exponent at the base is strictly negative so that the concavity of the fibre maps implies that
$\lim_{n\to\infty}f_t^n(\theta,x)=0$ for a.e. $\theta\in\Theta$ w.r.t. the volume measure and all $x\in\R_\geqslant$.
\end{remark}

\begin{theorem}  \label{theo:A_dim}
Suppose that $\gamma_{\min} < \gamma_{\max}$ so that $\log g$ is not cohomologous to a constant. Then
  
  \[ \begin{array}{ll}
       \dim_H ( \mathcal{A} _t ) = \dim_P ( \mathcal{A} _t ) = 3 & \mbox{for } t \in \left( \gamma _{\min} , {\gamma_c^-}       \right] ,\\
       \dim_H ( \mathcal{A} _t ) = \dim_P ( \mathcal{A} _t ) = D ( t ) + 2 &
       \mbox{for } t \in \left[ {\gamma_c^-}, \gamma_{\max} \right) .
     \end{array} \]
\end{theorem}

\begin{remark}
Theorem~\ref{theo:D_dim} remains true also for Baker maps $B_a:[0,1]^2\to[0,1]^2$ for $ a \in ( 0 , 1 ) $,
\begin{equation}
 B_a(u, v) \assign \begin{cases}
     \left( \frac{u}{a} , a v \right)  & \mbox{for } u <     a\\
     \left(  \frac{u - a}{1-a} , a + ( 1 - a )v \right) &
     \mbox{for } u \geqslant a
   \end{cases} .
   \end{equation}
   In this case $\mu_{\mathrm{SRB}}^-=\mu_{\mathrm{SRB}}^+=m^2$ (the Lebesgue measure on $[0,1)^2$) so that $\gamma_c^-=\gamma_c^+=\int\log g\,dm^2$, and
\begin{displaymath}
\log \left\| d B_a \left| E^s \right. \right\|=\log \left( a \mathbbm{1}_{ [ 0, a ) \times [ 0, 1 ) } + ( 1 - a ) \mathbbm{1}_{ [ a, 1 ) \times  [ 0, 1 )} \right) .
\end{displaymath}   

The reason that the same (indeed, even much simpler) arguments apply is essentially that the discontinuity line is the boundary of a Markov partition.
So there are no additional difficulties passing from the symbolic to the (piecewise) smooth system, see also \cite{Schmeling2001}.
\end{remark}

\begin{remark}
For a baker map with $a=0.45$ and for $g(u,v)=1.001+\cos(2\pi v)$ we determined $D(t)$ numerically, see Figure~\ref{fig:1}. The seeming discontinuities at $\gamma_{\min}$ and $\gamma_{\max}$ are numerical artefacts due to the fact that only trajectories of length $N=21$ were used to approximate the equilibrium states $\mu_t$.

In order to see that the limits of $D(t)$ for $t\to\gamma_{\min/\max}$ are indeed zero, it suffices to show that all $\mu\in\mathcal{P}_T(\Theta)$ which minimize or maximize $\mu(\log g)$ have entropy $h_T(\mu)=0$, because then
$\lim_{t\to\gamma_{\min/\max}}h_T(\mu_t)=0$ due to the upper semicontinuity of $h_T$, and so also $\lim_{t\to\gamma_{\min/\max}}D(t)=0$, because the Lyapunov exponents of all $\mu_t$ are uniformly bounded away from zero. For $t\to\gamma_{\max}$ this yields a full proof, because it is easily seen 
that only the point masses in the fixed points $(0,0)$ and $(1,1)$ maximize $\mu(\log g)$, and so $\lim_{t\to\gamma_{\max}}D(t)=0$. For $t\to\gamma_{\min}$ the situation is less clear. There is numerical evidence (no proof) that the equidistribution on a 3-cycle minimizes $\mu(\log g)$. If this is indeed the case, then also $\lim_{t\to\gamma_{\min}}D(t)=0$. See also \cite{KJR2012} for more details.
\end{remark}

\begin{figure}\label{fig:1}
\includegraphics[scale=0.2]{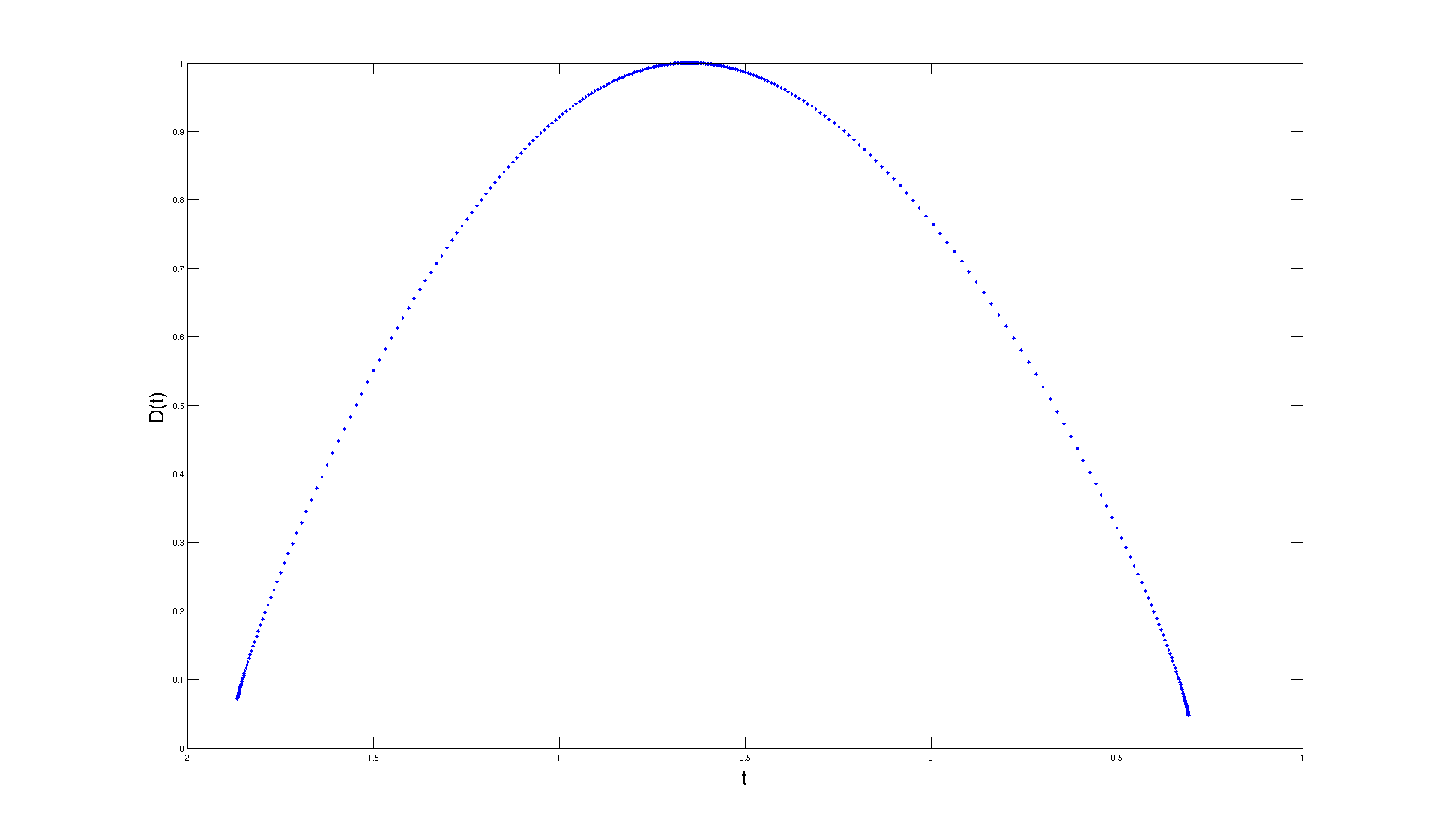}
\caption{The function $D(t)$ for a baker map with $a=0.45$ and $g(u,v)=1.001+\cos(2\pi v)$.}
\end{figure}

\section{Proofs\label{sec:proofs}}

\subsection{Proofs from section \ref{sec:characterization}}\label{subsec:proofs1}
  Let $ ( \psi_{t, n} )_{n\geqslant0}$ and $M_t > 0$ be as defined in Lemma
  \ref{max_invariant_function}, and denote $g_t \assign e^{- t} g$.

\begin{proof}[Proof of Theorem \ref{theo:main-Lambda}]
  \begin{enumerate}[a)]
  \item
Let $\theta \in \Theta$ such that $\Gamma_t \left( \theta \right) < 0$.
  There are $\delta > 0$ and integers $n_1<n_2<\dots$ such that
  \begin{equation}
   \frac{d}{d x} f_t^{n_j} \left( T^{- n_j} \theta, x \right) _{|x=0} = \prod_{i = 1}^{n_j } g_t \left( T^{- i} \theta \right) \leqslant e^{- \delta n_j}
\end{equation}
        for all $j\in\N$.
  As $\frac{d^2}{d x^2} f_{^{} t}^{n_j} \left( T^{- n_j} \theta, x \right) <
  0$ for all $x \in ( 0 , \infty ) $ (cf. Remark \ref{f_calculation}), we have
  \[ \varphi_t \left( \theta \right) \leqslant \psi_{t, n_j} \left( \theta\right) = f_t^{n_j} \left( T^{- n_j} \theta, M_t \right) \leqslant
     e^{- \delta n_j} M_t \longrightarrowlim^{j \rightarrow \infty} 0 \fs  \]
\item
Let $\theta \in \Theta$ such that $\Gamma_t \left( \theta \right) > 0$.
  There are $\delta > 0$ and $C > 0$ such that, for all $n \in \mathbbm{N}_0$,
  \begin{equation}
    \prod_{i = 1}^{n} g_t \left( T^{- i} \theta \right) \geqslant C
    e^{\delta n}\fs  \label{eq:point0}
  \end{equation}
  For a contradiction, suppose $\varphi_t \left( \theta \right) = 0$. Let
  $\varepsilon_0 \nocomma > 0$ and $n_0 \in \mathbbm{N}$ such that
  \begin{equation}
    \inf_{x \in \left( 0, \varepsilon_0 \right)} \frac{h \left( x \right)}{x}
    \geqslant e ^{-\delta} \hspace{1em} \text{\tmop{and}}
    \hspace{1em} \psi_{t, n _0} \left( \theta \right) <
    \varepsilon_0 \fs  \label{eq:point1}
  \end{equation}
  Let $k \left( n \right) \assign \max \left\{ k \leqslant n : f_t^k \left(
  T^{- n} \theta, M_t \right) \geqslant \varepsilon_0 \right\}$.
	As $k(n)<n$ if $n\geqslant n_0$, we have
for $j \in \left\{ 0,  \dots, n - k \left( n \right) - 1 \right\}$,
  i.e. for $n - j \in \left\{ k \left( n \right) + 1, \dots, n \right\}$,
  \begin{eqnarray}
    f_t^{n - k \left( n \right) - j} \left( T^{- n + k \left( n \right)}
    \theta, \varepsilon_0 \right) & \leqslant & f_t^{n - k \left( n \right) -
    j} \left( T^{- n + k \left( n \right)} \theta, f_t^{k \left( n \right)}
    \left( T^{- n} \theta, M_t \right) \right) \nonumber\\
    & = & f_t^{n - j} \left( T^{- n} \theta, M_t \right) \hspace{1em} < \;
    \varepsilon_0 \fs  \label{eq:point3}
  \end{eqnarray}
  Using now (\ref{eq:point3}) with $j=0$ and then (\ref{eq:point1}) and
  (\ref{eq:point3}) repeatedly, we obtain for $n \geqslant n_0$
  \begin{eqnarray}
    \psi_{t, n} \left( \theta \right) 
	& = &f_t^n(T^{-n}\theta,M_t)\nonumber\\    
    & \geqslant & f_t^{n - k \left( n
    \right)} \left( T^{- n_{} + k \left( n \right)} \theta, \varepsilon_0
    \right) \nonumber\\
    & = & g_t \left( T^{- 1} \theta \right) h \left( f_t^{n - k \left( n
    \right) - 1} \left( T^{- n + k \left( n \right)} \theta, \varepsilon_0
    \right) \right) \nonumber\\
    & \geqslant & g_t \left( T^{- 1} \theta \right) e ^{-\delta} f_t^{n - k \left( n \right) - 1} \left( T^{- n + k \left( n
    \right)} \theta, \varepsilon_0 \right) \nonumber\\
    & = & g_t \left( T^{- 1} \theta \right) e ^{-\delta}
    g_t \left( T^{- 2} \theta \right) h \left( f_t^{n - k \left( n \right) -
    2} \left( T^{- n + k \left( n \right)} \theta, \varepsilon_0 \right)
    \right) \nonumber\\
    & \geqslant & g_t \left( T^{- 1} \theta \right) e ^{-\delta} g_t \left( T^{- 2} \theta \right) e ^{-\delta}
    f_t^{n - k \left( n \right) - 2} \left( T^{- n + k \left( n \right)}
    \theta, \varepsilon_0 \right) \nonumber\\
    & \vdots &  \nonumber\\
    & \geqslant & \left( \prod_{i = 1}^{n - k \left( n \right)} g_t \left(
    T^{- i} \theta \right) \right) ( e ^{-\delta} )^{n - k
    \left( n \right)} \varepsilon_0 \fs   \label{eq:point4}
  \end{eqnarray}
  With (\ref{eq:point0}) and the first estimate of (\ref{eq:point1}) we
  obtain from (\ref{eq:point4})
  \[ \psi_{t, n} \left( \theta \right) \geqslant C e^{\delta \left( n - k
     \left( n \right) \right)}  e ^{-\delta ( n - k (
     n ))} \varepsilon_0 = C \varepsilon_0 \com \]
  and therefore $\varphi_t( \theta) \geqslant C \varepsilon_0$. This
  contradicts the assumption.
  \item This is a direct consequence of \ref{theo:main-Lambda-a}) and 
  \ref{theo:main-Lambda-b}).
  \end{enumerate}
\end{proof}

\begin{proof}[Proof of Proposition~\ref{prop:SN}.]\quad\\
Let $\theta\in\Theta\setminus N_t$, i.e. $\varphi_t(\theta)>0$.
By definition of the functions $\psi_{t,n}$ we have
\[ 
\psi_{t, n} (\theta) = 
 g_t(T^{- 1} \theta)\,h (\psi_{t, n - 1} (T^{- 1} \theta))
\]
  so that
  \begin{equation}
   \log \psi_{t, n} (\theta) = \log g_t (T^{- 1} \theta) + H (\psi_{t, n
     - 1} (T^{- 1} \theta)) + \log \psi_{t, n - 1}  (T^{- 1} \theta)   
\end{equation}
  where $H (x) \assign \log (x^{- 1} h (x))$. As $h$ is concave
  and $h'(0)=1$, $H$ extends to a decreasing function from $[0,
  \infty)$ to $(- \infty, 0]$ with $H (0) = 0$ and $H'(0)=\frac{1}{2}h''(0)<0$. Iterating this identity
  $\ell$ times $(1\leqslant\ell\leqslant n)$ yields
  \begin{equation}
    \log \psi_{t, n} (\theta) = \sum_{k = 1}^\ell \log g_t
    (T^{- k} \theta) + \sum_{k = 1}^\ell H (\psi_{t, n - k} (T^{- k} \theta)) +
    \log \psi_{t,n-\ell}(T^{-\ell}\theta) \fs 
    \label{eq:summation}
  \end{equation}
  For fixed $\ell\in\N$ we get in the limit $n\to\infty$
  \begin{equation}\label{eq:H-rec}
      \log \varphi_{t} (\theta) = \sum_{k = 1}^\ell \log g_t
    (T^{- k} \theta) + \sum_{k = 1}^\ell H (\varphi_{t} (T^{- k} \theta)) +
    \log \varphi_{t}(T^{-\ell}\theta) \fs 
  \end{equation}

Recall that $\theta\in\Theta\setminus N_t$. Because of Theorem~\ref{theo:main-Lambda}\ref{theo:main-Lambda-c}), $\theta\in S_t$ if and only if 
$\sum_{k = 1}^{\ell_i} \log g_t(T^{- k} \theta)=o(\ell_i)$ along some subsequence $(\ell_i)_{i\in\N}$. In  view of (\ref{eq:H-rec}) this is equivalent to
\begin{equation}
\label{eq:equiv-1}
  \sum_{k = 1}^{\ell_i} |H (\varphi_{t} (T^{- k} \theta))|
  +|\log \varphi_{t}(T^{-\ell_i}\theta)|
  =
  o(\ell_i) \fs 
\end{equation}
As $H'(0)<0$, there is a constant $C>0$ such that
  \begin{equation}
  C^{-1}\varphi_{t} (T^{- k} \theta)
  \leqslant
  |H (\varphi_{t} (T^{- k} \theta))|
  \leqslant
  C\varphi_{t} (T^{- k} \theta)\quad(k\geqslant0) \com
  \end{equation}
  so that (\ref{eq:equiv-1}) is equivalent to
\begin{equation}
  \sum_{k = 1}^{\ell_i} \varphi_{t} (T^{- k} \theta)
  +|\log \varphi_{t}(T^{-\ell_i}\theta)|
  =
  o(\ell_i) \fs 
\end{equation}
\end{proof}

\subsection{Proof of Theorem \ref{theo:D_dim}}

\begin{remark}\label{remark:Anosov} We refer to \cite{Bowen,Keller_book,Parry/Pollicott} for the following well-known facts about topologically mixing Anosov diffeomorphisms:
\begin{enumerate}[\quad$\triangleright$]
\item The entropy function $ h _T ( \cdot ) : \mathcal{P} _T ( \Theta ) \to \R_\geqslant $ is upper semi-continuous, as $ T $ is expansive.
 \item For Hölder-continuous functions $ \phi , \psi : \Theta \rightarrow \R $ the function $ t \mapsto P ( t \phi + \psi) $ is real analytic and strictly convex, unless $ \phi $ is cohomologous to a constant, where $ P (t \phi+\psi ) $ denotes the topological pressure of $ T $ for the potential $ t\phi+\psi $.
\end{enumerate}
\end{remark}

We shall prepare three Lemmas, before we prove the theorem.
In this section we assume that $ \gamma _{ \min } < \gamma _{ \max } $.
Recall that we defined $D(t)$ as
\begin{equation} 
 	 D ( t ) = \max \left\{ \frac{h_T \left( \mu \right)}{\mu
     (-\log \| d T | E^s  \|)} : \mu \in
     \mathcal{P}_T ( \Theta ) \mbox{ and } \gamma(\mu) = t \right\} .
  \end{equation}

\begin{lemma}\label{lemma:dim_D+1}
$\dim_H(S_t')=D(t)+1$ for $t\in(\gamma_{\min},\gamma_{\max})$.
Furthermore, there is a real analytic function $ q : ( \gamma_{\min}, \gamma_{\max} ) \rightarrow \mathbbm{R}$ such that for the unique equilibrium measure $\mu_t \in \mathcal{E}_T ( \Theta  )$ of the potential $ q ( t ) \log g_t - D ( t )  \log \| d T | E^s \|$ the following holds:
\begin{equation}
\mu_t ( S_t ) = 1,\quad D ( t )  = \frac{h_T ( \mu_t )}{\mu_t (-\log \| d T | E^s \| )},\quad\text{and}\quad
\dim_H(\mu_t)=D(t)+1.
\end{equation}
\end{lemma}
\begin{proof}
For $\theta\in\Theta$ denote by $V^s_{loc}(\theta)$ the (suitably defined) local stable manifold of $T$ through $\theta$. Then \cite[Theorem 12.2.2]{Barreira} asserts that $\dim_H(S_t')=\dim_H(S_t'\cap V^s_{loc}(\theta))+1$ for each $\theta\in S_t'$. With the same argument as in \cite[Example 7.2.5]{Barreira} one shows that $\dim_H(S_t'\cap V^s_{loc}(\theta))$ coincides with the $u$-dimension of $S_t'$ for the function $u=-\log\|dT|E^s\|$, and 
\cite[Theorem 10.1.4]{Barreira} asserts that this $u$-dimension takes the value $D(t)$. (For the application of this last theorem observe that
all Hölder potentials have a unique equilibrium state.)
The assertions on $q(t)$ and $\mu_t$ follow from \cite[Theorems 10.1.4 and 10.3.1]{Barreira} together with Corollary \ref{coro:S_ident}, \cite[Lemma 10.1.6]{Barreira} and its proof.
\end{proof}

\begin{lemma} \label{D_analytic}
 $D : ( \gamma_{\min}, \gamma_{\max} )  \rightarrow [ 0, 1 ]$ is a real analytic function such that $D (  {\gamma_c^-} ) = 1$, $D''( {\gamma_c^-} ) < 0 $ and
  \[ D' ( t ) = \left\{ \begin{array}{ll}
       > 0 & \mbox{for } t \in ( \gamma_{\min}, {\gamma_c^-} ) \\
       < 0 & \mbox{for } t \in ( {\gamma_c^-}, \gamma_{\max} )
     \end{array} \right. . \]
 
\end{lemma}

\begin{proof}[Proof of Lemma \ref{D_analytic}]
For this proof, we will extend the proof of \cite[Theorem 10.3.1]{Barreira}:
Denote $( u , \Phi ) \assign ( - \log\|dT|E^s\| , \log g ) $ and consider the real analytic function $Q \left( \delta, q, t \right) \assign P \left( q \left( \Phi - t \right) -
  \delta u \right)$. In that proof it is shown that the equations
  \begin{equation}
    \left\{ \begin{array}{c}
      Q \left( D \left( t \right), q \left( t \right), t \right) = 0\\
      \frac{\partial Q}{\partial q} \left( D \left( t \right), q \left( t
      \right), t \right) = 0
    \end{array} \right.  \label{eq:equations0}
  \end{equation}
  determine a real analytic function $\left( D, q \right) : \left( \gamma_{\min},
  \gamma_{\max} \right) \rightarrow \mathbbm{R}^2$.
  These equations are further equivalent to
    \begin{equation}
    \left\{ \begin{array}{l}
      h _T ( \mu _t ) + \mu _t (q ( t ) \Phi - D ( t )  u  ) = q ( t ) t\\
      \mu _t ( \Phi ) = t
    \end{array}  \right. . \label{eq:equations2}
  \end{equation}
  Differentiating the first equality of (\ref{eq:equations0}) with respect to $t$ and using the
  second one, we obtain
  \begin{equation}
    D' \left( t \right) = \frac{q \left( t \right)}{\frac{\partial
    Q}{\partial \delta} \left( D( t ), q( t ), t
    \right)} \fs  \label{eq:D_deriv}
  \end{equation}
  Similarly, differentiating the second equality of (\ref{eq:equations0}) and using (\ref{eq:D_deriv}), we obtain
  \begin{equation}
    q'( t ) = \frac{1}{\frac{\partial^2 Q}{\partial q^2} \left(
    D ( t ), q ( t ), t \right)} \left( 1 - q ( t ) \frac{\frac{\partial^2 Q}{\partial \delta \partial q} \left( D( t ), q ( t ), t \right)}{\frac{\partial Q}{\partial
    \delta} \left( D ( t ), q ( t ), t \right)} \right) .
    \label{eq:q_deriv}
  \end{equation}
  
From the second equality of (\ref{eq:equations2}) it follows $\mu_t \left( S_t' \right) = 1$ by Birkhoff's theorem, whence $\mu_t \left( S_t \right) = 1 $ in view of Corollary \ref{coro:S_ident}.
  Moreover, from (\ref{eq:equations2}) we obtain
    \begin{equation}
    D( t ) = \frac{h_T \left( \mu_t \right)}{\mu_t ( u )} \fs  \label{D_frac}
  \end{equation}

  Furthermore, from $\left( \ref{eq:SRB} \right.$) we have
  \begin{equation}
  \left\{
  	\begin{array}{l}
  	 Q \left( 1, 0, {\gamma_c^-} \right) = P \left( - u \right) = h_T \left(
     \mu_{\mathrm{SRB}}^- \right) - \mu_{\mathrm{SRB}}^- \left( u \right) = 0 \\  
   \frac{\partial Q}{\partial q} \left( 1, 0, {\gamma_c^-} \right) = 
     \mu_{\mathrm{SRB}}^- \left( \varphi \right) - {\gamma_c^-} = 0
     \end{array}
     \right. 
  \end{equation}
 so that
 \begin{equation}
 \left( 1, 0, {\gamma_c^-} \right) = \left( D ( {\gamma_c^-}), q({\gamma_c^-}), {\gamma_c^-} \right).	\label{eq:D_stabil}
 \end{equation}	
	In view of (\ref{eq:D_deriv}), we have also
	\begin{equation}
	D' \left( {\gamma_c^-} \right) = 0	\fs 	\label{eq:D_deriv_zero}
	\end{equation}
  
  Finally, differentiating (\ref{eq:D_deriv}) we obtain
  \begin{equation}
	D''(t)=\frac{q'(t)}{\frac{\partial Q}{\partial \delta}}
	-q(t)\,\frac{D'(t)\frac{\partial^2Q}{\partial\delta^2}+q'(t)\frac{\partial^2Q}{\partial\delta \partial q}+\frac{\partial^2Q}{\partial\delta \partial t}}{\left(\frac{\partial Q}{\partial\delta}\right)^2}
  \end{equation}
  where all partial derivatives are evaluated at $(D(t),q(t),t)$.
  Observing that $\frac{\partial^2Q}{\partial\delta \partial t}=0$ and substituting (\ref{eq:D_deriv}) we obtain
  \begin{equation}
  D''( t ) = \frac{1}{\frac{\partial Q}{\partial \delta}
     } \left( q'( t) - D'( t ) \left( D'(t)\,\frac{\partial^2Q}{\partial \delta ^2 } + q'( t)    \frac{\partial^2 Q}{\partial \delta \partial q}\right) \right) 	\label{eq:D_2deriv}
  \end{equation}
  with all partial derivatives again evaluated at $(D(t),q(t),t)$.
  As we assume that $ \gamma _{ \min } < \gamma _{ \max } $,
the map $ q \mapsto Q ( \delta , q , t ) $ is strictly convex (c.f. Remark \ref{remark:Anosov}).
  Hence, we obtain from (\ref{eq:q_deriv}) and (\ref{eq:D_2deriv}) together with (\ref{eq:D_stabil}) and (\ref{eq:D_deriv_zero}),
  \begin{equation}
    q' \left( {\gamma_c^-} \right) = \frac{1}{\frac{\partial^2 Q}{\partial q^2}
    \left( 1, 0, \gamma_c^- \right)} > 0 \hspace{1em} \text{\tmop{and}} \hspace{1em} D''
    \left( {\gamma_c^-} \right) = \frac{q' \left( {\gamma_c^-} \right)}{\frac{\partial
    Q}{\partial \delta} \left( 1, 0,\gamma_c^- \right)} < 0 \fs  \label{eq:qD}
  \end{equation}
  
  Now consider the curve $C : \left( \gamma_{\min}, \gamma_{\max} \right)
  \rightarrow \left[ 0, 1 \right] \times \mathbbm{R}$ defined by $C( t) =\left( D( t), q( t ) \right)$. In view of
  (\ref{eq:D_stabil}) and the first inequality of (\ref{eq:qD}) this curve
  runs across the point $\left( 1, 0 \right)$ in $t = {\gamma_c^-}$ from left to right.
  As $\tmop{sign} \left( D'(t)\right)= - \tmop{sign} \left( q ( t )\right)$ by (\ref{eq:D_deriv}), $D'(t)$ can change its sign only when $q(t)=0$, i.e. only if $t=\gamma_c^-$ and hence $D(t)=1$. Therefore $D'(t)>0$ for $t<\gamma_c^-$ and $D'(t)<0$ for $t>\gamma_c^-$.
\end{proof}

Recall that for $ \nu \in \mathcal{P} ( \Theta ) $ the upper and lower point-wise dimensions at $ \theta \in \Theta $ are defined by
\begin{equation}
	\overline{d} _\nu ( \theta ) \assign \limsup _{ r \rightarrow 0 } \frac{ \log \nu ( B ( \theta , r ) ) }{ \log r }
\quad
\mbox{and}
\quad
	\underline{d} _\nu ( \theta ) \assign \liminf _{ r \rightarrow 0 } \frac{ \log \nu ( B ( \theta , r ) ) }{ \log r } \fs 
\end{equation}
If $ \overline{d} _\nu ( \theta ) = \underline{d} _\nu ( \theta ) $, we denote this with $ d _\nu ( \theta ) $.

To formulate the next lemma we define the sets
\begin{equation}
	N_t^{\geqslant} \assign \left\{ \theta \in \Theta : \Gamma \left( \theta \right)
   \geqslant t \right\} \quad 
   \mbox{and}	\quad
   N_t^{\leqslant} \assign \left\{ \theta \in \Theta : \limsup _{n \rightarrow \infty} \Gamma ^{(n)} \left( \theta \right) \leqslant t \right\}   .
\end{equation}
   
\begin{lemma}
  \label{lem:D_point_dim}
Let $t\in(\gamma_{\min},\gamma_{\max})$.  
  There is $\nu_t \in \mathcal{P} \left( \Theta \right)$ (which is in general not $T$-invariant) such that
  $\nu_t \left( S_t' \right) = 1$ with the following properties:
\begin{enumerate}
 \item $ d_{\nu_t} \left( \theta \right) = D \left( t \right) + 1 $ for $ \nu _t $-a.e. $ \theta \in \Theta $,
 \item $ \overline{d}_{\nu_t} \left( \theta \right) \leqslant D \left( t \right) + 1 $ for each $ \theta \in S _t ' $,
 \item $ \underline{d}_{\nu_t} \left( \theta \right) \leqslant D \left( t \right) + 1 $ for each $ \theta \in R _t ' $,
 \item $ \overline{d}_{\nu_t} \left( \theta \right) \leqslant D \left( t \right) + 1 $ for each $ \theta \in N _t ^\leqslant $
 provided that $ t<\gamma _c^-$, and
 \item $ \overline{d}_{\nu_t} \left( \theta \right) \leqslant D \left( t \right) + 1 $ for each $ \theta \in N _t ^\geqslant $ provided that $ t> \gamma _c^-$.
\end{enumerate}
\end{lemma}

\begin{proof}[Proof of Lemma \ref{lem:D_point_dim}]
We apply the very general result \cite[Theorem 12.3.1]{Barreira} with slight modifications. Here is a dictionary between our notation and the notation from \cite{Barreira}:
\begin{displaymath}
\begin{array}{c|c}
\text{Notation from \cite{Barreira}}&\text{Our notation}\\
\hline
\Lambda,\ f,\ \mathcal{M}&\Theta,\ T,\ \mathcal{P}_T(\Theta)\\
\kappa&1\\
\mathcal{P}^+ ( \, \cdot \, ) ,\  \mathcal{P}^- ( \, \cdot \, )  &  1 ,\ \gamma ( \cdot ) \\
\Phi^+,\ \Phi^-,\ \Psi^+,\ \Psi^- & 1 ,\ \log g ,\ 1,\ 1\\
\alpha,\ \beta & 1,\ t\\
K_\alpha^+,\ K_\beta^-& \Theta,\ S_t'\\
d^+,\ d^-&1,\ D(t)\\
q^+,\ q^-& 1,\ q(t)
\end{array}
\end{displaymath}
We note that the assumption $\alpha\in\operatorname{int}\mathcal{P}^+(\mathcal{M})$ of \cite[Theorem 12.3.1]{Barreira} is not satisfied in our setting, as $\mathcal{P}^+(\mathcal{M})=\{1\}=\{\alpha\}$. This assumption is only used to assure the existence of $q^+$ with the properties claimed in
\cite[Lemma 12.3.3]{Barreira}; but these properties are trivially satisfied in our case.

Now we can conclude immediately from \cite[Theorem 12.3.1]{Barreira} and Lemma~\ref{lemma:dim_D+1} that
\begin{equation}
\nu _t (S_t')=\nu _t (K_\alpha^+\cap K_\beta^-)=1 \com
\end{equation}
and
\begin{equation}
d_{\nu_t} \left( \theta \right) = 
D \left( t \right) + 1\quad\text{for $ \nu _t $-a.e. $ \theta \in \Theta $} \com
\end{equation}
\begin{equation}
 \overline{d}_{\nu_t} \left( \theta \right) \leqslant D \left( t \right) + 1 
	\quad\text{for each $ \theta \in S _t ' $} \com
\end{equation}
as
\begin{equation}
\dim_H(K_\alpha^+)+\dim_H(K_\beta^-)-\dim_H(\Lambda)=
\dim_H(S_t')=D \left( t \right) + 1 \fs 
\end{equation}
For the remaining proofs of assertions (3) to (5) we must modify some arguments from \cite{Barreira} slightly:

\noindent To (3) : We modify the proof of \cite[Lemma 12.3.6]{Barreira} as follows.
Since inequality (12.15) does no longer hold for all large $n$'s, but still for infinitely many $n$'s, several inequalities after the estimate (12.17) hold only for $r$'s such that $m \left( \theta, r \right)$ defined in (12.12) satisfies (12.15) in place of $ n $.
Choose a null sequence $ ( r_k ) _k$ of $r$'s in the above manner.
Then we obtain the last two inequalities in the proof of Lemma 12.3.6 where both limits superior are replaced by the limits inferior.

\noindent To (4) : We
  modify the proof of \cite[Lemma 12.3.6]{Barreira} again. Let $t \in \left(\gamma_{\min},{\gamma_c^-}\right)$ such that $q(t)<0$. Then,
  for $\delta > 0$ and $\chi(\omega) \in N_t^{\leqslant}$, there is a $r (\omega) \in \N $
  such that for $n > r \left( \omega \right) $
  \begin{equation}\label{eq:To(4)}
  q(t) \cdot \sum_{k = 0}^n \left( \Phi ^s \left( \left(
     \sigma^- \right)^k \omega ^- \right) - t \right) > 
     - \delta n | q(t)| . 
  \end{equation}
  Since we can obtain (12.17) from this estimate instead of (12.15), the
  proof is finished.
  
\noindent To (5) : Let $t \in \left( {\gamma_c^-}, \gamma_{\max} \right)$ such that $q(t)>0$. As in the proof of (4), 
 for $\delta > 0$ and $\chi \left(
  \omega \right) \in N_t^{\geqslant}$ there is a $r \left( \omega \right) \in \N $
  such that (\ref{eq:To(4)}) holds for $n > r(\omega)$, and again we
 can obtain (12.17) from this estimate instead of (12.15).
\end{proof}

\begin{lemma}
  \label{S'_D}For $t \in \left( \gamma_{\min}, \gamma_{\max} \right)$ we have $ \dim_H \left( S_t' \right) = \dim_H \left( R_t' \right) = \dim_P \left(
     S_t' \right) = D \left( t \right) + 1$.
Furthermore, 
	$\dim_H(N_t^\leqslant)=D(t)+1$ for $t\in\left(\gamma_{\min},\gamma_c^-\right]$ and  
  $\dim_P \left( N_t^{\geqslant} \right) = D \left( t \right) +
  1$ for $t \in \left[ {\gamma_c^-}, \gamma_{\max} \right)$.
\end{lemma}

\begin{proof}[Proof of Lemma \ref{S'_D}]
  This follows from Lemma \ref{lem:D_point_dim} with \cite[Theorem 2.1.5]{Barreira} and \cite[Proposition 2.3]{Falconer2}.
\end{proof}

\begin{remark}
Roughly speaking, each of the sets $ S _t ' $ and $ R _t ' $ is locally the product of a $ D ( t ) $-dimensional subset of the local stable manifold and the complete local unstable manifold.
\end{remark}

\begin{proof}[Proof of the Theorem \ref{theo:D_dim}]
  As $S_t' \subseteq S_t \subseteq R_t'$, we obtain $\dim_H \left( S_t \right) = D
  \left( t \right) + 1$ for $t \in \left( \gamma_{\min}, \gamma_{\max}
  \right)$ from Lemma \ref{S'_D}. 
  For the remaining arguments observe also the monotonicity properties of $t\mapsto D(t)$ from Lemma~\ref{D_analytic}.
  
  Let $t \in \left( \gamma_{\min}, {\gamma_c^-} \right]$.
  Firstly, $D(t) + 1 = \lim_{t' \nearrow t} D( t') + 1 = \lim_{t'\nearrow t} \dim_H( S_{t'}) \leqslant \dim_H( N_t)$, as $S_{t'} \subseteq N_t$.
  Secondly, from Theorem \ref{theo:main-Lambda} we have $N_t \subseteq \left\{ \Gamma \leqslant t \right\} \subseteq R '_t\cup N ^{\leqslant} _t $.
  Thus, $\dim_H( N_t) \leqslant \max \{ \dim_H( R'_t ) , \dim _H (N ^{\leqslant} _t) \} =  D( t ) + 1$, and it follows that $\dim_H( N_t ) = D( t) + 1$.
  
  Now, let $t \in \left[ {\gamma_c^-}, \gamma_{\max} \right)$. Firstly, $D( t ) + 1 = \lim_{t'' \searrow t} D ( t ) + 1 =
  \lim_{t'' \searrow t} \dim_H( S_{t''} ) \leqslant \dim_H ( \Theta \setminus N_t)$, as $S_{t''} \subseteq \Theta \setminus N_t$.
  Secondly, from Theorem \ref{theo:main-Lambda} we have $\Theta \setminus N_t \subseteq N ^\geqslant _t $.
  Thus, $\dim_H( \Theta \setminus N_t)\leqslant\dim_P( \Theta \setminus N_t ) \leqslant \dim_P( N_t^{\geqslant}) = D( t) + 1$, and it follows that $\dim_H ( \Theta \setminus N_t) = \dim_P( \Theta \setminus N_t) = D( t) + 1$.
\end{proof}

\subsection{Proof of Theorem \ref{theo:A_dim}}

\begin{proof}
\noindent To (2): Let $ t \in[ \gamma _c ^- , \gamma _{\max} ) $. For $ t' \in ( t , \gamma _{\max} ) $ and $ \nu _{t'} \in \mathcal{P} ( \Theta ) $ from Lemma \ref{lem:D_point_dim} it follows $ \nu _{t'} \left( \bigcup _{ k \in \N } \{ \varphi _t > 1/k \} \right) =\nu _{t'} ( \Theta \setminus N _t ) \geqslant \nu _{t'} ( S _{t'} ) = 1 $, whence there is a $ k \in \N $ s.t. $ \nu _{t'} ( \{ \varphi _t > 1 / k \} ) > 0 $.
Denoting the restriction of $ \nu _{t'} $ to this set by $ \tilde{\nu} _{t'} $ we obtain $ \dim _H ( \{ \varphi _t > 1 / k \} ) \geqslant \dim _H ( \tilde{\nu} _{t'} ) \geqslant D ( t' ) + 1 $ by \cite[Theorem 2.1.5]{Barreira}, as $ d_{ \tilde{\nu} _{t'}} = d _{\nu _{t'}} = D(t') + 1 $ holds $ \tilde{\nu} _{t'}$-a.e..
Therefore, $\dim _H ( \{ \varphi _t > 1 / k \} ) \geqslant \lim _{t' \searrow t} D ( t' ) + 1 = D( t ) + 1 $, and it follows from the first product rule of \cite[Theorem 3]{Tricot} that
\[
	\dim _H ( \mathcal{A} _t ) \geqslant \dim _H \left(  \{ \varphi _t > 1 / k \} \times [ 0 , 1/k ] \right) \geqslant \dim _H (  \{ \varphi _t > 1 / k \} ) + \dim _H ( [ 0 , 1/k ] ) \geqslant D ( t ) + 2 .
\]
 On the other hand, Theorem \ref{theo:D_dim} in conjunction with the last product rule of \cite[Theorem 3]{Tricot} implies
\[
	\dim _P ( \mathcal{A} _t ) \leqslant \dim _P \left( \Theta \setminus N _t \times [ 0 , M _t ] \right) \leqslant \dim _P ( \Theta \setminus N _t ) + \dim _P ( [ 0 , M _t ] ) = D ( t ) + 2 ,
\]
so that $ \dim _P ( \mathcal{A} _t ) = \dim _H ( \mathcal{A} _t ) = D ( t ) + 2 $.

\noindent To (1): Let $ t \in ( \gamma _{ \min } , \gamma _c ^- ] $.
From (2) and Theorem \ref{theo:D_dim} it follows $ \dim _P ( \mathcal{A} _t ) \geqslant \dim _H ( \mathcal{A} _t ) \geqslant \dim _H ( \mathcal{A} _{ \gamma _c ^-} ) = D ( \gamma _c ^- ) + 1 = 3 $, which finishes the proof.
\end{proof}

\bibliography{public10}

\begin{thebibliography}{10}

\bibitem{Arnold-book}
L.~Arnold.
\newblock {\em Random Dynamical Systems}.
\newblock Springer, 2 edition, 2003.

\bibitem{Barreira}
L.~Barreira.
\newblock {\em Dimension and recurrence in hyperbolic dynamics}.
\newblock Birkh{\"a}user, 2008.

\bibitem{Bowen}
R.~Bowen.
\newblock {\em Equilibrium states and the ergodic theory of Anosov
  diffeomorphisms}.
\newblock Springer, 2 edition, 2007.

\bibitem{Falconer2}
K.~Falconer.
\newblock {\em Techniques in Fractal Geometry}.
\newblock Wiley, 1997.

\bibitem{Keller1996}
G.~Keller.
\newblock A note on strange nonchaotic attractors.
\newblock {\em Fund. Math}, 151:139--148, 1996.

\bibitem{Keller_book}
G.~Keller.
\newblock {\em Equilibrium States in Ergodic Theory}.
\newblock Cambridge University Press, 1998.

\bibitem{KJR2012}
G.~Keller, H.~Jaffri, and R.~Ramaswamy.
\newblock Generalized synchronozation in chaotically driven systems.
\newblock {\em In preparation}, 2012.

\bibitem{Parry/Pollicott}
W.~Parry and M.~Pollicott.
\newblock {\em Zeta functions and the periodic orbit structure of hyperbolic
  dynamics}.
\newblock Soc. Math. de France, 1990.

\bibitem{Schmeling2001}
J.~Schmeling.
\newblock Entropy preservation under markov coding.
\newblock {\em Journal of Statistical Physics}, 104:799--815, 2001.

\bibitem{Sturman/Stark}
R.~Sturman and J.~Stark.
\newblock Semi-uniform ergodic theorems and applications to forced systems.
\newblock {\em Nonlinearity}, 13:113, 2000.

\bibitem{Tricot}
C.~Tricot.
\newblock Two definitions of fractional dimension.
\newblock {\em Math. Proc. Camb. Phil. Soc.}, 91:57--74, 1982.

\end{thebibliography}
\bibliographystyle{abbrv}

\end{document}